\documentclass[12pt,reqno]{amsart}
\usepackage[left=3cm,top=2cm,right=3cm,bottom=2cm]{geometry}
\usepackage{amssymb}
\usepackage{amsbsy}
\usepackage{epsfig}
\usepackage{verbatim}
\usepackage[normalem]{ulem}
\usepackage[pdftex]{hyperref}
\hypersetup{colorlinks=true,linkcolor=blue,citecolor=blue,breaklinks = true}

\usepackage{tikz}
\tikzstyle{vertex} = [fill,shape=circle,node distance=80pt]
\tikzstyle{edge} = [fill,opacity=.5,fill opacity=.5,line cap=round, line join=round, line width=40pt]
\tikzstyle{elabel} =  [fill,shape=circle,node distance=30pt]

\pgfdeclarelayer{background}
\pgfsetlayers{background,main}

\begin{document}
\title{Line Multigraphs of General Hypergraphs}
 	
\author[K. Cardoso]{Kau\^e Cardoso} \address{Instituto Federal do Rio Grande do Sul - Campus Feliz, Feliz, RS, Brazil.}
\email{\tt kaue.cardoso@feliz.ifrs.edu.br}


\pdfpagewidth 8.5 in \pdfpageheight 11 in

\newcommand{\h}{\mathcal{H}}
\newcommand{\g}{\mathcal{G}}
\newcommand{\A}{\mathbf{A}}
\newcommand{\B}{\mathbf{B}}
\newcommand{\C}{\mathbf{C}}
\newcommand{\D}{\mathbf{D}}
\newcommand{\M}{\mathbf{M}}
\newcommand{\N}{\mathbf{N}}
\newcommand{\lin}{\mathcal{L}}
\newcommand{\cli}{\mathcal{C}}
\newcommand{\Q}{\mathbf{Q}}
\newcommand{\x}{\mathbf{x}}
\newcommand{\y}{\mathbf{y}}
\newcommand{\z}{\mathbf{z}}
\newcommand{\E}{\mathtt{E}}
\newcommand{\Ah}{\mathbf{A}(\mathcal{H})}

\theoremstyle{plain}
\newtheorem{theorem}{Theorem}[section]
\newtheorem{lemma}[theorem]{Lemma}
\newtheorem{proposition}[theorem]{Proposition}
\newtheorem{corollary}[theorem]{Corollary}

\theoremstyle{definition}
\newtheorem{definition}[theorem]{Definition}
\newtheorem{claim}[theorem]{Claim}
\newtheorem{question}[theorem]{Question}
\newtheorem{example}[theorem]{Example}
\newtheorem{remark}[theorem]{Remark}

\newcommand{\keyword}[1]{\textsf{#1}}

\begin{abstract}
A line multigraph is obtained from a hypergraph by taking its hyperedges as vertices and joining two of them by as many edges as the number of vertices they share. We develop a matrix theory for line multigraphs of general, not necessarily uniform, hypergraphs. The central tool is the identity $\mathbf{B}^\mathrm{T}\mathbf{B} = \mathbf{C} + \mathbf{A}_{\mathcal{L}}$, where $\mathbf{B}$ is the incidence matrix, $\mathbf{C}$ is the diagonal matrix of hyperedge cardinalities and $\mathbf{A}_{\mathcal{L}}$ is the adjacency matrix of the line multigtaph. From this identity, we prove that the eigenvalues of the line multigraph of a hypergraph of rank $r$ are at least $-r$, and we describe the eigenspace and the multiplicity of $-r$ through an essential core of the hypergraph. We also give an explicit combinatorial condition under which $-r$ is attained. As applications, we bound the spectral radius of the signless Laplacian matrix, characterizing the cases of equality, and we determine the complete signless Laplacian spectrum of a general power hypergraph. On the structural side, we show that connectivity, linearity, and regularity transfer between a hypergraph and its line multigraph, that every hypergraph shares its line multigraph with infinitely many others, and that each class of such hypergraphs contains a reduced representative.
\newline

\newline

\end{abstract}

\maketitle

\section{Introduction}

In spectral graph theory, structural properties of a graph are studied through the eigenvalues and eigenvectors of a matrix associated with it, most often the adjacency matrix. Extending this study to hypergraphs, where a hyperedge may join more than two vertices, is more difficult. A natural higher-order analogue is the adjacency tensor introduced by Cooper and Dutle \cite{Cooper}, but computing tensor eigenvalues is hard, both in theory and in practice \cite{NP-hard}. For this reason, many authors have studied matrix representations of a hypergraph, which encode less structure than the tensor but allow the use of classical linear algebra \cite{Banerjee-matriz, kaue-lap, Feng, kr-regular, estrada-index, Rodriguez1, unified-matrix-I, unified-matrix-II, distance, extremal-hypertree}. In this paper, we develop one such representation, the \textit{line multigraph}, into a theory that applies to hypergraphs with hyperedges of arbitrary sizes.

Given a hypergraph $\mathcal{H}$, its line multigraph $\mathcal{L}(\mathcal{H})$ records how the hyperedges of $\mathcal{H}$ overlap: each hyperedge becomes a vertex, and two such vertices are joined by as many edges as the number of vertices the corresponding hyperedges share. This construction is closely related to the classical line graph of a graph \cite{line-1,line-1932,line-0} and to intersection graphs of hypergraphs \cite{line-3-uniform, line-survei-arx, line, line-survei}. The two notions coincide on linear hypergraphs; in general, however, the intersection graph discards edge multiplicities, while the line multigraph keeps them. We introduced the line multigraph, together with its basic spectral properties, in \cite{kaue-energia,kaue-lap}, but only for uniform hypergraphs, in which every hyperedge has the same cardinality. This restriction is significant: many natural hypergraphs are not uniform, and the tools of these two previous papers do not apply to them directly. One purpose of this paper is to remove this restriction.

Our approach is based on the incidence matrix $\mathbf{B}(\mathcal{H})$ of a hypergraph, whose $(v,e)$ entry records whether the vertex $v$ belongs to the hyperedge $e$. For a $k$-uniform hypergraph, it is known that $\mathbf{B}^\mathrm{T}\mathbf{B} = k\mathbf{I} + \mathbf{A}_{\mathcal{L}}$, where $\mathbf{A}_{\mathcal{L}}$ is the adjacency matrix of the line multigraph \cite{kaue-lap}, and this single identity supports essentially all of the uniform theory. When hyperedges may have different sizes, the scalar matrix $k\mathbf{I}$ must be replaced by the diagonal matrix $\mathbf{C}$ of hyperedge cardinalities, which gives the identity $\mathbf{B}^\mathrm{T}\mathbf{B} = \mathbf{C} + \mathbf{A}_{\mathcal{L}}$ (Theorem \ref{teo:multigrafo}). Most spectral results in this paper are, in some way, a consequence of this identity.

On the structural side, we show that connectivity, linearity, and regularity transfer between a hypergraph and its line multigraph. In particular, $\mathcal{L}(\mathcal{H})$ is regular exactly when $\mathcal{H}$ satisfies a new condition that we call skew edge-regularity, which agrees with the edge-regularity of \cite{kaue-sign-vector} on uniform hypergraphs (Theorem \ref{teo:arestareg}). We also show that the map $\mathcal{H} \mapsto \mathcal{L}(\mathcal{H})$ is far from injective: every connected multigraph in which each vertex has at least two distinct neighbors is the line multigraph of some hypergraph (Proposition \ref{prop:all_multigraphs}), and every hypergraph shares its line multigraph with infinitely many others (Remark \ref{rem:redundancy}).

On the spectral side, the identity $\mathbf{B}^\mathrm{T}\mathbf{B} = \mathbf{C} + \mathbf{A}_{\mathcal{L}}$ immediately implies that every eigenvalue of $\mathcal{L}(\mathcal{H})$ is at least $-r$, where $r$ is the rank of $\mathcal{H}$ (Proposition \ref{teo:lower_bound}). Beyond this bound, the incidence matrix approach gives a complete description of the eigenspace of $-r$. We isolate an \textit{essential core} of $\mathcal{H}$, obtained by repeatedly discarding hyperedges that cannot support an eigenvector, and we show that this core determines both the eigenspace and its exact dimension (Theorem \ref{teo:multiplicity}). We also introduce \textit{collars}, a hypergraph analogue of even cycles, and we show that the presence of a collar is a purely combinatorial sufficient condition for $-r$ to be an eigenvalue (Theorem \ref{teo:collar_eigenvalue} and Corollary \ref{cor:collar_general}).

These structural and spectral tools are applied in Section \ref{sec:App} to the signless Laplacian matrix $\mathbf{Q}(\mathcal{H}) = \mathbf{B}\mathbf{B}^\mathrm{T}$. We bound its spectral radius in terms of the rank and the co-rank of $\mathcal{H}$, and we characterize the cases of equality (Theorems \ref{teo:raio_espectral} and \ref{teo:qcotas}), extending results previously known only for uniform hypergraphs \cite{kaue-sign-vector, kaue-lap}. Power hypergraphs have been studied from both the matrix and the tensor point of view \cite{Kaue-power,Khan,Lucas-integral}, but little was known in the non-uniform case; we determine the complete signless Laplacian spectrum of a general power hypergraph (Theorem \ref{teo:power_spectrum}), which had previously been computed only when the base hypergraph is uniform \cite{kaue-lap}.

\medskip
\noindent\textbf{Relation to previous work.} We now state precisely what is old, meaning already contained in our own earlier work on uniform hypergraphs, and what is new, meaning not considered before even in the uniform case. For uniform hypergraphs, the following results were established in \cite{kaue-lap}: the incidence matrix identity of Theorem \ref{teo:multigrafo}, the resulting lower bound $-k$ (Proposition \ref{teo:lower_bound}), the degree formula for line multigraphs (Lemma \ref{lem:graulinha}), the equality in Theorem \ref{teo:raio_espectral}, the bounds of Theorem \ref{teo:qcotas}, and the signless Laplacian spectrum of power hypergraphs with uniform base (Theorem \ref{teo:power_spectrum}). The scaling of line multigraphs under power operations (Lemma \ref{lem:linhk} and Proposition \ref{teo:espec-power-lin}) was established, for the uniform case, in \cite{kaue-energia}. The notion of edge-regularity, and the characterization of equality in the bounds of Theorem \ref{teo:qcotas}, were given, again for the uniform case, in \cite{kaue-sign-vector}. All of these results were stated only for uniform hypergraphs. Every other result in this paper is new even in the uniform case, and none of it was previously available for general hypergraphs. The generalizations from the uniform case to the general case presented in this paper often lead to proofs that are shorter and more elegant than the original ones, because the theory of line multigraphs developed here condenses information about the hypergraph that earlier proofs had to assemble separately in each case.

The remainder of the paper is organized as follows. Section \ref{sec:pre} fixes notation and basic definitions. Section \ref{sec:structural} develops the structural theory summarized above. Section \ref{sec:incidence} introduces the incidence identity and its spectral consequences. Section \ref{sec:App} applies this theory to the signless Laplacian matrix. Section \ref{sec:conclusion} presents a short discussion and some open problems raised in the paper.

\section{Preliminaries}\label{sec:pre}

In this section, we introduce the fundamental definitions, terminology, and notation for hypergraphs and line multigraphs. This framework provides the foundation for the structural and spectral results in subsequent sections. Illustrative examples are included to clarify key concepts. For a more comprehensive treatment, we refer the reader to \cite{Bretto,kaue-lap}.

\begin{definition}
A \textit{hypergraph} $\mathcal{H}=(V,E)$ consists of a finite set of vertices $V(\mathcal{H})$ and a set of hyperedges $E(\mathcal{H})$, where each hyperedge is a non-empty subset of $V$. A hypergraph is called \textit{simple} if it contains no hyperedges of cardinality one and no hyperedge is contained in another. Throughout this work, we consider exclusively simple hypergraphs; therefore, for brevity, the term simple will be omitted hereafter. The \textit{rank} and the \textit{co-rank} of a hypergraph are defined as the largest and smallest cardinality of its hyperedges, respectively. For $k \ge 2$, a hypergraph $\mathcal{H}$ is \textit{$k$-uniform} if all its hyperedges have cardinality $k$. In this sense, a simple graph is precisely a $2$-uniform hypergraph.
\end{definition}

\begin{definition}
A \textit{multigraph} $\mathcal{G}=(V,E)$ is defined by a finite set of vertices $V(\mathcal{G})$ and a multiset of edges $E(\mathcal{G})$, where each edge consists of a set containing two vertices. For a hypergraph $\mathcal{H}$, its \textit{line multigraph} $\mathcal{L}(\mathcal{H})$ is the multigraph where $V(\mathcal{L}(\mathcal{H}))=E(\mathcal{H})$. The number of edges (multiplicity) connecting two distinct vertices $u, v \in V(\mathcal{L}(\mathcal{H}))$ is equal to the cardinality of the intersection between their corresponding hyperedges $e_u, e_v \in E(\mathcal{H})$; that is, the multiplicity is given by $|e_u \cap e_v|$.
\end{definition}

\begin{example}
In Figure \ref{fig:multi-lin}, we illustrate a hypergraph $\mathcal{H}$ with vertex set $V(\mathcal{H})=\{1,2,3,4,5\}$ and hyperedge set $E(\mathcal{H})=\{e_1, e_2, e_3\}$, where $e_1=\{1,2,3\}$, $e_2=\{1,4,5\}$, and $e_3=\{3,4,5\}$. Its corresponding line multigraph $\mathcal{L}(\mathcal{H})$ is also shown. Note that since $|e_2 \cap e_3| = 2$, there are two edges connecting the vertices $145$ and $345$ in $\mathcal{L}(\mathcal{H})$.

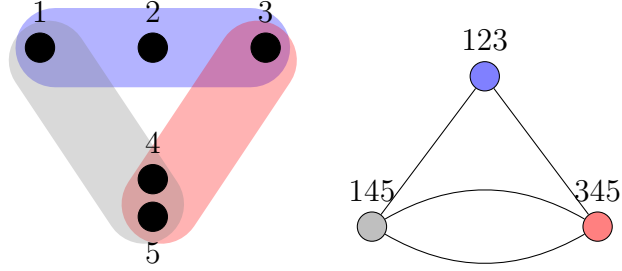
\begin{figure}[!h]	
	\centering	
	\begin{tikzpicture}
		\node[draw,circle,fill=black,label=above:\(1\)] (v1) at (0,2) {};
		\node[draw,circle,fill=black,label=above:\(2\)] (v2) at (1.5,2) {};
		\node[draw,circle,fill=black,label=above:\(3\)] (v3) at (3,2) {};
		\node[draw,circle,fill=black,label=above:\(4\)] (v4) at (1.5,0.25) {};
		\node[draw,circle,fill=black,label=below:\(5\)] (v5) at (1.5,-0.25) {};
		
		\begin{pgfonlayer}{background}
			\draw[edge,color=blue,line width=30pt, opacity=0.3] (v1) -- (v3);
			\draw[edge,color=red,line width=30pt, opacity=0.3] (v3) -- (v5);
			\draw[edge,color=gray,line width=30pt, opacity=0.3] (v5) -- (v1);			
		\end{pgfonlayer}				
	\end{tikzpicture}
	\quad
	\begin{tikzpicture}	
		\node[draw,circle,fill=blue!50, label=above:\(123\)] (a) at (1.5, 2) {};
		\node[draw,circle,fill=gray!50, label=above:\(145\)] (b) at (0, 0) {};
		\node[draw,circle,fill=red!50, label=above:\(345\)] (c) at (3, 0) {};
		
		\path
		(a) edge (b)
		(a) edge (c)
		(b) edge [bend left] (c)
		(b) edge [bend right] (c);
	\end{tikzpicture}
	\caption{The hypergraph $\mathcal{H}$ and its line multigraph $\mathcal{L}(\mathcal{H})$.}\label{fig:multi-lin}
\end{figure}  
\end{example}

\begin{definition}
Let $\mathcal{H}$ be a hypergraph or a multigraph. The \textit{(hyper)edge neighborhood} of a vertex $v \in V$, denoted by $E_{[v]}$, is the set of all (hyper)edges containing $v$. The \textit{degree} of a vertex $v \in V$, denoted by $d(v)$, is the number of (hyper)edges containing $v$; specifically, $d(v) = |E_{[v]}|$. A vertex is said to be \textit{isolated} if its degree is zero. To avoid ambiguity when dealing with multiple structures simultaneously, we denote the degree as $d_{\mathcal{H}}(v)$. The \textit{maximum}, \textit{minimum}, and \textit{average degrees} of $\mathcal{H}$ are defined, respectively, as:
$$\Delta(\mathcal{H}) = \max_{v \in V}\{d(v)\}, \quad \delta(\mathcal{H}) = \min_{v \in V}\{d(v)\}, \quad \text{and} \quad d(\mathcal{H}) = \frac{1}{|V|}\sum_{v \in V}d(v).$$
\end{definition}

\begin{definition}
In either a multigraph or a hypergraph, a \textit{path} is an alternating sequence of distinct vertices and (hyper)edges such that each (hyper)edge contains the vertices immediately preceding and following it in the sequence. A hypergraph or a multigraph is said to be \textit{connected} if there exists a path between every pair of vertices; otherwise, it is \textit{disconnected}.
\end{definition}

\begin{lemma}\label{lem:conect}
Let $\mathcal{H}$ be a hypergraph without isolated vertices. Then, $\mathcal{L}(\mathcal{H})$ is connected if and only if $\mathcal{H}$ is connected.
\end{lemma}
\begin{proof}
Suppose $\mathcal{H}$ is connected, and let $e, f \in E(\mathcal{H})$ be two distinct hyperedges. Choose vertices $u \in e$ and $w \in f$. Since $\mathcal{H}$ is connected, there is a path $u = v_0, e_1, v_1, \ldots, e_l, v_l = w$ in $\mathcal{H}$. Two consecutive hyperedges in this path share a vertex, so they are adjacent in $\mathcal{L}(\mathcal{H})$; moreover, $e$ shares the vertex $u$ with $e_1$, and $f$ shares the vertex $w$ with $e_l$. Concatenating these adjacencies yields a walk from $e$ to $f$ in $\mathcal{L}(\mathcal{H})$, and every walk contains a path between the same endpoints. Hence $\mathcal{L}(\mathcal{H})$ is connected.

Conversely, suppose $\mathcal{L}(\mathcal{H})$ is connected, and let $u, w \in V(\mathcal{H})$ be two distinct vertices. Since $\mathcal{H}$ has no isolated vertices, there exist hyperedges $e_0, e_l \in E(\mathcal{H})$ with $u \in e_0$ and $w \in e_l$. By hypothesis, there is a path $e_0, e_1, \ldots, e_l$ in $\mathcal{L}(\mathcal{H})$, and consecutive hyperedges in this path share at least one vertex, say $v_i \in e_{i-1} \cap e_i$. The sequence $u, e_0, v_1, e_1, v_2, \ldots, v_l, e_l, w$ is then a walk from $u$ to $w$ in $\mathcal{H}$, which again contains a path. Hence $\mathcal{H}$ is connected.
\end{proof}

Lemma \ref{lem:conect} establishes that connectivity is preserved between a hypergraph and its line multigraph. In this work, we consider only connected hypergraphs; consequently, their associated line multigraphs are also connected. To simplify the exposition, we shall henceforth omit explicit mention of connectivity and refer to these structures simply as hypergraphs.

\begin{definition}
A hypergraph $\mathcal{H}$ is said to be \textit{linear} if every pair of distinct hyperedges shares at most one vertex.
\end{definition}

\begin{lemma}
Let $\mathcal{H}$ be a hypergraph. The line multigraph $\mathcal{L}(\mathcal{H})$ is a simple graph if and only if $\mathcal{H}$ is linear.
\end{lemma}
\begin{proof}
Suppose $\mathcal{L}(\mathcal{H})$ is a simple graph. By definition, any two distinct vertices $v_1, v_2 \in V(\mathcal{L}(\mathcal{H}))$ are joined by at most one edge. Let $e_1, e_2 \in E(\mathcal{H})$ be the hyperedges corresponding to $v_1$ and $v_2$, respectively. Since the number of edges between $v_1$ and $v_2$ in $\mathcal{L}(\mathcal{H})$ is given by $|e_1 \cap e_2|$, it follows that $|e_1 \cap e_2| \leq 1$. Thus, $\mathcal{H}$ is linear.

Conversely, assume $\mathcal{H}$ is linear. By definition, any two distinct hyperedges $e_1, e_2 \in E(\mathcal{H})$ share at most one vertex; that is, $|e_1 \cap e_2| \leq 1$. Let $v_1, v_2 \in V(\mathcal{L}(\mathcal{H}))$ be the vertices in the line multigraph representing $e_1$ and $e_2$. The multiplicity of the edge connecting $v_1$ and $v_2$ is $|e_1 \cap e_2|$, which is at most one. Therefore, $\mathcal{L}(\mathcal{H})$ contains no multiple edges and is, consequently, a simple graph.
\end{proof}

The definitions of the intersection graph (see \cite{line}) and the line multigraph of the hypergraph coincide for linear hypergraphs. More fundamentally, many properties depend only on the existence of adjacencies between vertices, not on edge multiplicities. Consequently, much of the established theory of intersection graphs extends directly to the setting adopted here. For this reason, we focus primarily on characteristics that specifically exploit edge multiplicities.

\section{Structural properties of line multigraphs}\label{sec:structural}

This section studies how the structural properties of a hypergraph are reflected in its line multigraph: which properties transfer unchanged, which are transformed, and how much of the hypergraph can be recovered from its line multigraph. We first study degrees and regularity. We then introduce collars, a family of hypergraphs that will control the extremal spectral behavior in Section \ref{sec:incidence}, and we examine how far the map $\mathcal{H} \mapsto \mathcal{L}(\mathcal{H})$ is from being injective.

\begin{lemma}\label{lem:graulinha}
Let $\mathcal{H}$ be a hypergraph and $\mathcal{L}(\mathcal{H})$ be its line multigraph. If $u \in V(\mathcal{L}(\mathcal{H}))$ is the vertex corresponding to the hyperedge $e_u \in E(\mathcal{H})$, then:
	$$d_{\mathcal{L}}(u) = \left( \sum_{v \in e_u} d_{\mathcal{H}}(v) \right) - |e_u|.$$
\end{lemma}
\begin{proof}
For each vertex $v \in e_u$, there are exactly $d_{\mathcal{H}}(v) - 1$ hyperedges, other than $e_u$, that contain $v$. By the construction of the line multigraph $\mathcal{L}(\mathcal{H})$, each of these hyperedges corresponds to a vertex adjacent to $u$ via an edge generated by the shared vertex $v$. By summing the contributions of every vertex $v \in e_u$, we conclude that the degree of the vertex $u \in V(\mathcal{L}(\mathcal{H}))$ is:
$d_{\mathcal{L}}(u) = \sum_{v \in e_u} (d_{\mathcal{H}}(v) - 1) = \left( \sum_{v \in e_u} d_{\mathcal{H}}(v) \right) - |e_u|.$
\end{proof}

\begin{definition}\label{def:regular}
A hypergraph $\mathcal{H}$ is \textit{regular} if there exists a constant $d$ such that $d_{\mathcal{H}}(v) = d$ for all $v \in V(\mathcal{H})$.
\end{definition}

The concept of edge-regularity was introduced as an edge-based analogue to traditional vertex regularity \cite{kaue-sign-vector}. A hypergraph $\mathcal{H}$ is said to be \textit{edge-regular} if for every hyperedge $e \in E(\mathcal{H})$, the sum of the degrees of its incident vertices is invariant. To extend the utility of this property to non-uniform structures, we propose the following generalization:

\begin{definition}\label{def:skew_regular}
A hypergraph $\mathcal{H}$ is \textit{skew edge-regular} if there exists a constant $D$ such that:
$$\sum_{v \in e} (d_{\mathcal{H}}(v) - 1) = D, \quad \forall e \in E(\mathcal{H}).$$
\end{definition}

This modification stems from analyzing incidence patterns in general hypergraphs. The sum of $(d(v) - 1)$ normalizes the contribution of each hyperedge based on its intersection potential, eliminating distortions caused by degree-one vertices. The following lemma demonstrates that for uniform hypergraphs, these two definitions are perfectly consistent.

\begin{lemma}\label{lem:uniform_equiv}
Let $\mathcal{H}$ be a $k$-uniform hypergraph. Then $\mathcal{H}$ is edge-regular if and only if $\mathcal{H}$ is skew edge-regular.
\end{lemma}
\begin{proof}
Since $\mathcal{H}$ is $k$-uniform, $|e| = k$ for every $e \in E(\mathcal{H})$. The result follows immediately from the identity $\sum_{v \in e} (d_{\mathcal{H}}(v) - 1) = \sum_{v \in e} d_{\mathcal{H}}(v) - k$. Since $k$ is constant, the sum of degrees $\sum d_{\mathcal{H}}(v)$ is invariant across all hyperedges if and only if the sum $\sum (d_{\mathcal{H}}(v) - 1)$ is also invariant.
\end{proof}

\begin{example}
For graphs (2-uniform hypergraphs), edge-regularity and skew edge-regularity coincide (Lemma \ref{lem:uniform_equiv}), and both are equivalent to the graph being either regular or semi-regular bipartite (Example 4.10 in \cite{kaue-sign-vector}). These are precisely the 2-uniform structures with constant degree sums across edges.
\end{example}

\begin{theorem}\label{teo:arestareg}
Let $\mathcal{H}$ be a hypergraph. The line multigraph $\mathcal{L}(\mathcal{H})$ is regular if and only if $\mathcal{H}$ is skew edge-regular.
\end{theorem}

\begin{proof}
By Lemma \ref{lem:graulinha}, the degree of any vertex $u \in V(\mathcal{L}(\mathcal{H}))$ is given by $d_{\mathcal{L}}(u) = \sum_{v \in e} (d_{\mathcal{H}}(v) - 1)$, where $e$ is the hyperedge corresponding to $u$. It follows that $d_{\mathcal{L}}(u)$ is a constant for all vertices in the line multigraph if and only if the sum $\sum_{v \in e} (d_{\mathcal{H}}(v) - 1)$ is invariant across all hyperedges in $\mathcal{H}$. By Definition \ref{def:skew_regular}, this is precisely the condition for $\mathcal{H}$ to be skew edge-regular.
\end{proof}

\begin{remark}
If $\mathcal{H}$ is a $d$-regular hypergraph, the degree of a vertex $u \in V(\mathcal{L}(\mathcal{H}))$ corresponding to the hyperedge $e$ simplifies to $d_{\mathcal{L}}(u) = |e|(d-1)$. This follows directly from Lemma \ref{lem:graulinha} by substituting $d_{\mathcal{H}}(v) = d$ for all $v \in e$. In particular, if $\mathcal{H}$ is both $k$-uniform and $d$-regular, its line multigraph $\mathcal{L}(\mathcal{H})$ is $k(d-1)$-regular.
\end{remark}

\begin{definition}\label{def:edge_coloring}
For a hypergraph $\mathcal{H}$, a mapping $f:E(\mathcal{H})\rightarrow\{1,\dots,p\}$ is an \textit{edge $p$-coloring} if, for every pair of adjacent hyperedges $e_i, e_j \in E(\mathcal{H})$, we have $f(e_i) \neq f(e_j)$.
\end{definition}

\begin{definition}\label{def:collar}
A hypergraph $\mathcal{H}$ is a \textit{collar} if it is $2$-regular and admits an edge $2$-coloring $f:E(\mathcal{H})\rightarrow\{1,2\}$. For each hyperedge $e \in E(\mathcal{H})$, if $f(e)=1$, we say $e$ is an \textit{odd edge}; if $f(e)=2$, we say $e$ is an \textit{even edge}. Consequently, every vertex $v \in V(\mathcal{H})$ is incident to exactly one even edge and one odd edge.
\end{definition}

\begin{example}
A connected graph is a collar if and only if it is an even cycle. 
Figure \ref{fig:collar} illustrates a $3$-uniform collar.

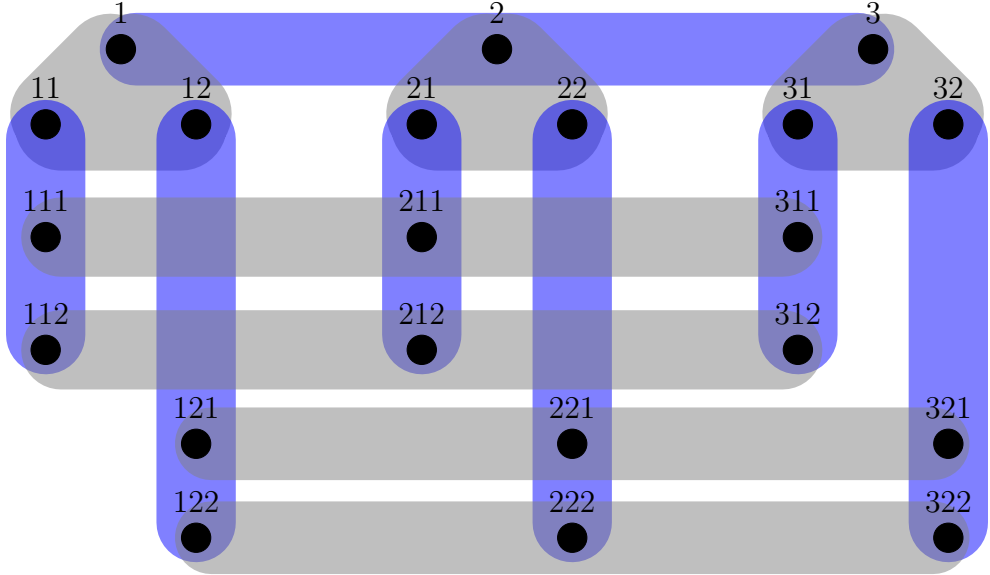
\begin{figure}[h]
	\centering		
	\begin{tikzpicture}
		\node[draw,circle,fill=black,label=above:\(1\)] (1) at (0,-2.5) {};
		\node[draw,circle,fill=black,label=above:\(2\)] (2) at (5,-2.5) {};
		\node[draw,circle,fill=black,label=above:\(3\)] (3) at (10,-2.5) {};
		\node[draw,circle,fill=black,label=above:\(11\)] (11) at (-1,-3.5) {};
		\node[draw,circle,fill=black,label=above:\(12\)] (12) at (1,-3.5) {};
		\node[draw,circle,fill=black,label=above:\(21\)] (21) at (4,-3.5) {};
		\node[draw,circle,fill=black,label=above:\(22\)] (22) at (6,-3.5) {};
		\node[draw,circle,fill=black,label=above:\(31\)] (31) at (9,-3.5) {};
		\node[draw,circle,fill=black,label=above:\(32\)] (32) at (11,-3.5) {};
		\node[draw,circle,fill=black,label=above:\(111\)] (111) at (-1,-5) {};
		\node[draw,circle,fill=black,label=above:\(112\)] (112) at (-1,-6.5) {};
		\node[draw,circle,fill=black,label=above:\(121\)] (121) at (1,-7.75) {};
		\node[draw,circle,fill=black,label=above:\(122\)] (122) at (1,-9) {};
		\node[draw,circle,fill=black,label=above:\(211\)] (211) at (4,-5) {};
		\node[draw,circle,fill=black,label=above:\(212\)] (212) at (4,-6.5) {};
		\node[draw,circle,fill=black,label=above:\(221\)] (221) at (6,-7.75) {};
		\node[draw,circle,fill=black,label=above:\(222\)] (222) at (6,-9) {};
		\node[draw,circle,fill=black,label=above:\(311\)] (311) at (9,-5) {};
		\node[draw,circle,fill=black,label=above:\(312\)] (312) at (9,-6.5) {};
		\node[draw,circle,fill=black,label=above:\(321\)] (321) at (11,-7.75) {};
		\node[draw,circle,fill=black,label=above:\(322\)] (322) at (11,-9) {};
		
		\begin{pgfonlayer}{background}
			\draw[edge,color=blue,line width=27.5pt] (1) -- (3);
			\draw[edge,color=gray,line width=35pt] (1) -- (11) -- (12) -- (1);
			\draw[edge,color=gray,line width=35pt] (2) -- (21) -- (22) -- (2);
			\draw[edge,color=gray,line width=35pt] (3) -- (31) -- (32) -- (3);	
			\draw[edge,color=blue,line width=30pt] (11) -- (112);
			\draw[edge,color=blue,line width=30pt] (12) -- (122);
			\draw[edge,color=blue,line width=30pt] (21) -- (212);
			\draw[edge,color=blue,line width=30pt] (22) -- (222);
			\draw[edge,color=blue,line width=30pt] (31) -- (312);
			\draw[edge,color=blue,line width=30pt] (32) -- (322);	
			\draw[edge,color=gray,line width=30pt] (111) -- (311);
			\draw[edge,color=gray,line width=30pt] (112) -- (312);
			\draw[edge,color=gray,line width=27.5pt] (121) -- (321);
			\draw[edge,color=gray,line width=27.5pt] (122) -- (322);
		\end{pgfonlayer}
		
	\end{tikzpicture}	
		\caption{~A $3$-uniform collar.}\label{fig:collar}
	\end{figure}
\end{example}

\begin{theorem}\label{teo:collar_bipartite}
Let $\mathcal{H}$ be a hypergraph. If $\mathcal{H}$ is a collar, then its line multigraph $\mathcal{L}(\mathcal{H})$ is bipartite.
\end{theorem}
\begin{proof}
By the definition of a collar, there exists an edge $2$-coloring $f: E(\mathcal{H}) \to \{1, 2\}$. Let $U$ and $W$ be the sets of odd and even hyperedges, respectively. Since $f$ is an edge $2$-coloring, no two hyperedges within $U$ (or $W$) are adjacent in $\mathcal{H}$. As the vertex set of $\mathcal{L}(\mathcal{H})$ corresponds to $E(\mathcal{H})$ and adjacency is preserved, there are no edges within $U$ or $W$ in the line multigraph. Thus, $\mathcal{L}(\mathcal{H})$ is bipartite.
\end{proof}

\begin{remark}\label{rem:collar_uniform}
In particular, if $\mathcal{H}$ is a $k$-uniform collar, then $\mathcal{L}(\mathcal{H})$ is both bipartite and $k$-regular. This regularity is a consequence of the collar being $2$-regular by definition; substituting $|e|=k$ and $d=2$ into the degree formula $d_{\mathcal{L}}(u) = |e|(d-1)$, we obtain $d_{\mathcal{L}}(u) = k(2-1) = k$ for every vertex $u$.
\end{remark}

\begin{proposition}\label{prop:all_multigraphs}
Let $\mathcal{G}$ be a multigraph with maximum degree $\Delta$ in which every vertex has at least two distinct neighbors. Then there exists a hypergraph $\mathcal{H}$ with rank $r = \Delta$ such that $\mathcal{L}(\mathcal{H}) = \mathcal{G}$.
\end{proposition}
\begin{proof}
Let $\mathcal{G} = (V, E)$ be a multigraph. We construct the hypergraph $\mathcal{H}$ by defining its vertex set as $V(\mathcal{H}) = E(\mathcal{G})$ and its hyperedge set as $E(\mathcal{H}) = \{e_u : u \in V(\mathcal{G})\}$, where each hyperedge $e_u$ consists of the edges in $\mathcal{G}$ that are incident to vertex $u$. By this construction, the cardinality of each hyperedge is $|e_u| = d_{\mathcal{G}}(u)$, which implies that $\text{rank}(\mathcal{H}) = \Delta(\mathcal{G})$. Moreover, the number of vertices shared between two distinct hyperedges $e_u$ and $e_v$ in $\mathcal{H}$ is, by definition, the number of edges connecting $u$ and $v$ in $\mathcal{G}$. Thus, both the adjacency and the multiplicity in $\mathcal{L}(\mathcal{H})$ perfectly coincide with those of $\mathcal{G}$, yielding $\mathcal{L}(\mathcal{H}) = \mathcal{G}$.

It remains to verify that $\mathcal{H}$ is simple. Since every vertex $u$ has at least two distinct neighbors, we have $|e_u| = d_{\mathcal{G}}(u) \geq 2$, so no hyperedge has cardinality one. Moreover, a containment $e_u \subseteq e_w$ with $u \neq w$ would mean that every edge incident to $u$ is also incident to $w$; that is, $w$ would be the unique neighbor of $u$, contradicting the hypothesis.
\end{proof}

\begin{remark}
Proposition \ref{prop:all_multigraphs} shows that the class of line multigraphs is remarkably rich: every multigraph in which each vertex has at least two distinct neighbors is realizable. This contrasts sharply with the restrictive forbidden-subgraph characterizations of classical line graphs \cite{line-1}, and indicates that the object of study here is the correspondence $\mathcal{H} \leftrightarrow \mathcal{L}(\mathcal{H})$, rather than the class of line multigraphs itself.
\end{remark}

\begin{lemma}\label{lem:remove-vertex}
Let $\mathcal{H}$ be a hypergraph. Suppose there exists a hyperedge $e \in E(\mathcal{H})$ such that $|e| \geq 3$ and there exists a vertex $v \in e$ with $d_{\mathcal{H}}(v) = 1$, and suppose further that $e \setminus \{v\}$ is not contained in any other hyperedge of $\mathcal{H}$. Let $\mathcal{H}'$ be the hypergraph obtained by removing $v$, defined by $V(\mathcal{H}') = V(\mathcal{H}) \setminus \{v\}$ and $E(\mathcal{H}') = (E(\mathcal{H}) \setminus \{e\}) \cup \{e \setminus \{v\}\}$. Then, $\mathcal{L}(\mathcal{H}) = \mathcal{L}(\mathcal{H}')$.
\end{lemma}

\begin{proof}
By definition, the vertex set $V(\mathcal{L}(\mathcal{H}))$ is in bijective correspondence with the edge set $E(\mathcal{H})$. Since the transformation from $\mathcal{H}$ to $\mathcal{H}'$ only modifies the internal composition of the hyperedge $e$ without removing the edge itself, $|E(\mathcal{H})| = |E(\mathcal{H}')|$, and we may identify $V(\mathcal{L}(\mathcal{H}))$ with $V(\mathcal{L}(\mathcal{H}'))$.

The multiplicity of the edge connecting any two vertices $u_i, u_j \in V(\mathcal{L}(\mathcal{H}))$ is given by the cardinality of the intersection of their corresponding hyperedges $e_i, e_j \in E(\mathcal{H})$. Because $d_{\mathcal{H}}(v) = 1$, the vertex $v$ belongs exclusively to $e$. Consequently, $v \notin e_j$ for any $e_j \neq e$, which implies that $e \cap e_j = (e \setminus \{v\}) \cap e_j$. Since no other hyperedge intersections are affected by the removal of $v$, all edge multiplicities are preserved. Thus, $\mathcal{L}(\mathcal{H}) = \mathcal{L}(\mathcal{H}')$.

Finally, $\mathcal{H}'$ is simple: the modified hyperedge has cardinality $|e| - 1 \geq 2$ and is not contained in any other hyperedge, while no hyperedge $e_j \neq e$ can be contained in $e \setminus \{v\} \subset e$, since $\mathcal{H}$ is simple.
\end{proof}

\begin{corollary}\label{cor:uniform_representation}
For any hypergraph $\mathcal{H}$ with rank $r$, and for all $k \geq r$, there exists a $k$-uniform hypergraph $\mathcal{H}'$ such that $\mathcal{L}(\mathcal{H}) = \mathcal{L}(\mathcal{H}')$. 
\end{corollary}
\begin{proof}
We construct $\mathcal{H}'$ from $\mathcal{H}$ by expanding each hyperedge to cardinality $k$. For every hyperedge $e \in E(\mathcal{H})$, we add $k - |e|$ new distinct vertices to $e$, defining each of these new vertices to have degree one. Reading Lemma \ref{lem:remove-vertex} in the opposite direction, a degree-one vertex belongs to no hyperedge other than its own, so it contributes to no intersection $e_i \cap e_j$ with $i \neq j$; consequently, the resulting line multigraph remains invariant. Thus, $\mathcal{H}'$ is $k$-uniform and satisfies $\mathcal{L}(\mathcal{H}) = \mathcal{L}(\mathcal{H}')$.
\end{proof}

\begin{remark}\label{rem:redundancy}
The invariance of $\mathcal{L}(\mathcal{H})$ under Lemma \ref{lem:remove-vertex} works in both directions. Removing a degree-one vertex whenever possible (\textit{reduction}) simplifies a hypergraph without changing its line multigraph. In the opposite direction, adding degree-one vertices can be repeated indefinitely: adjoining $p$ new degree-one vertices to any fixed hyperedge $e \in E(\mathcal{H})$ leaves $\mathcal{L}(\mathcal{H})$ unchanged for every $p \in \mathbb{N}$, so every hypergraph already shares its line multigraph with infinitely many others. Padding also allows any hypergraph to be represented as $k$-uniform for every $k$ at least as large as its rank (\textit{uniformization}); this is the mechanism behind the reductions of Section \ref{sec:App}.
\end{remark}

We now describe this redundancy through an equivalence relation on hypergraphs, and we ask how much of it is explained by Lemma \ref{lem:remove-vertex}.

\begin{definition}\label{def:line_equiv}
Two hypergraphs $\mathcal{H}_1$ and $\mathcal{H}_2$ are \textit{line-equivalent}, written $\mathcal{H}_1 \sim \mathcal{H}_2$, if there exists a bijection $\varphi: E(\mathcal{H}_1) \to E(\mathcal{H}_2)$ such that $|e \cap f| = |\varphi(e) \cap \varphi(f)|$ for all $e, f \in E(\mathcal{H}_1)$. Equivalently, $\varphi$ extends to an isomorphism of the multigraphs $\mathcal{L}(\mathcal{H}_1)$ and $\mathcal{L}(\mathcal{H}_2)$.
\end{definition}

Reflexivity and symmetry are immediate, and transitivity follows by composing the corresponding bijections; thus $\sim$ is an equivalence relation on hypergraphs. By Remark \ref{rem:redundancy}, every equivalence class containing at least one hyperedge is infinite, and by Corollary \ref{cor:uniform_representation} it contains a uniform representative of every rank $k$ at least as large as the rank of any of its members. This raises the natural question of what happens in the other direction: how small can a representative of a given class be made, and is there a canonical smallest one?

\begin{definition}\label{def:reduced}
A hypergraph $\mathcal{H}$ is \textit{reduced} if no hyperedge $e \in E(\mathcal{H})$ with $|e| \geq 3$ admits a vertex $v \in e$ with $d_{\mathcal{H}}(v) = 1$ such that $e \setminus \{v\}$ is not contained in any other hyperedge of $\mathcal{H}$; that is, if no application of Lemma \ref{lem:remove-vertex} is possible on $\mathcal{H}$.
\end{definition}

\begin{theorem}\label{teo:reduced_exists}
Every hypergraph $\mathcal{H}$ is line-equivalent to a reduced hypergraph.
\end{theorem}
\begin{proof}
If $\mathcal{H}$ is not reduced, Lemma \ref{lem:remove-vertex} produces a hypergraph $\mathcal{H}'$ with $\mathcal{L}(\mathcal{H}) = \mathcal{L}(\mathcal{H}')$ and $|V(\mathcal{H}')| = |V(\mathcal{H})| - 1$. Iterating, we obtain a sequence of hypergraphs with strictly decreasing, non-negative vertex count, all line-equivalent to $\mathcal{H}$ via the natural bijection that sends the modified hyperedge $e$ to $e \setminus \{v\}$ and fixes all others; the sequence must therefore terminate after finitely many steps, at which point the resulting hypergraph is reduced by construction.
\end{proof}

Theorem \ref{teo:reduced_exists} guarantees that every equivalence class contains a representative that cannot be simplified further by Lemma \ref{lem:remove-vertex}. The theorem does not say whether this representative is unique up to isomorphism, nor whether it depends on the order in which the reductions are performed. Indeed, removing a degree-one vertex from one hyperedge can change which hyperedges are contained in which, and this may block or allow a later reduction elsewhere. This question is the hypergraph counterpart of Whitney's isomorphism theorem for line graphs \cite{line-1932}, which states that two connected simple graphs with isomorphic line graphs are themselves isomorphic, with a single exceptional pair (the triangle $K_3$ and the star $K_{1,3}$).

\section{Spectral properties of line multigraphs}\label{sec:incidence}

This section studies the spectrum of the line multigraph through its relationship with the incidence matrix of the underlying hypergraph. Our results generalize classical properties of the eigenvalues of line graphs \cite{line-spec}, and they complement the decomposition-based lower bounds of Cioabă, Elzinga and Gregory \cite{smallest-eigenvalue} and of Knox and Mohar \cite{knox-mohar}.

\begin{definition}
Let $\mathcal{H} = (V, E)$ be a hypergraph. The \textit{incidence matrix} of $\mathcal{H}$, denoted by $\mathbf{B}(\mathcal{H})$, is a $|V| \times |E|$ matrix with entries $b_{ve}$ defined as:
$$b_{ve} = 
\begin{cases} 
1, & \text{if } v \in e, \\ 
0, & \text{otherwise}. 
\end{cases}$$ 

\noindent The \textit{cardinality matrix} $\mathbf{C}(\mathcal{H})$ is the $|E| \times |E|$ diagonal matrix where the $i$-th diagonal entry is given by $c_{ii} = |e_i|$, representing the cardinality of the hyperedge $e_i \in E$. 
\end{definition}

\begin{definition}
Let $\mathcal{G} = (V, E)$ be a multigraph. Its \textit{adjacency matrix} $\mathbf{A}(\mathcal{G})$ is a square matrix of order $|V|$ where the diagonal entries are $a_{vv} = 0$. For $v \neq u$, the entry $a_{vu}$ denotes the number of edges incident to both vertices $u$ and $v$ in $\mathcal{G}$.
\end{definition}

\begin{theorem}\label{teo:multigrafo}
Let $\mathcal{H}$ be a hypergraph with incidence matrix $\mathbf{B}$ and cardinality matrix $\mathbf{C}$. If $\mathbf{A}_{\mathcal{L}}$ is the adjacency matrix of its line multigraph $\mathcal{L}(\mathcal{H})$, then:
\begin{equation*}
\mathbf{B}^\mathrm{T}\mathbf{B} = \mathbf{C} + \mathbf{A}_{\mathcal{L}}.
\end{equation*}
\end{theorem}

\begin{proof}
Let $\mathbf{M} = \mathbf{B}^\mathrm{T}\mathbf{B}$. By the definition of matrix multiplication, the entry $m_{ij}$ is the inner product of the $i$-th and $j$-th columns of $\mathbf{B}$. Since these columns are the incidence vectors of hyperedges $e_i$ and $e_j$, respectively, $m_{ij}$ counts the number of vertices in the intersection $e_i \cap e_j$. We distinguish two cases:
\begin{itemize}
\item \textbf{Diagonal entries:} The entry $m_{ii}$ represents $|e_i \cap e_i| = |e_i|$. These entries form the diagonal cardinality matrix $\mathbf{C}$.

\item \textbf{Off-diagonal entries ($i \neq j$):} The entry $m_{ij} = |e_i \cap e_j|$. By the construction of the line multigraph, the number of edges between the vertices corresponding to $e_i$ and $e_j$ is precisely the number of shared vertices in $\mathcal{H}$. These entries form the adjacency matrix $\mathbf{A}_{\mathcal{L}}$.
\end{itemize}
Combining these cases, we obtain $\mathbf{M} = \mathbf{C} + \mathbf{A}_{\mathcal{L}}$.
\end{proof}

For $k$-uniform hypergraphs, the cardinality matrix is just a scalar matrix, $\mathbf{C} = k\mathbf{I}$, and Theorem \ref{teo:multigrafo} reduces to the identity established in \cite{kaue-lap}. The cardinality matrix is precisely what replaces this scalar matrix in the general setting.

\begin{proposition}\label{teo:lower_bound}
Let $\mathcal{H}$ be a hypergraph with rank $r$. If $\lambda$ is an eigenvalue of the adjacency matrix $\mathbf{A}_{\mathcal{L}}$ of the line multigraph $\mathcal{L}(\mathcal{H})$, then $\lambda \geq -r$.
\end{proposition}
\begin{proof}

Let $\mathbf{x}$ be a normalized eigenvector of $\mathbf{A}_{\mathcal{L}}$ associated with the eigenvalue $\lambda$. By Theorem \ref{teo:multigrafo}, we have the identity $\mathbf{B}^\mathrm{T}\mathbf{B} = \mathbf{C} + \mathbf{A}_{\mathcal{L}}$. Since $\mathbf{B}^\mathrm{T}\mathbf{B}$ is positive semi-definite, its quadratic form is non-negative:
\begin{equation*}
\mathbf{x}^\mathrm{T} (\mathbf{C} + \mathbf{A}_{\mathcal{L}}) \mathbf{x} \geq 0. 
\end{equation*}
Expanding this expression, we obtain:
\begin{equation*}
\mathbf{x}^\mathrm{T} \mathbf{C} \mathbf{x} + \mathbf{x}^\mathrm{T} \mathbf{A}_{\mathcal{L}} \mathbf{x} \geq 0 \implies \sum_{e_i \in E(\mathcal{H})} |e_i|x_i^2 + \lambda \geq 0.
\end{equation*}
By the definition of rank, $|e_i| \leq r$ for all $e_i \in E(\mathcal{H})$. Since $x_i^2 \geq 0$, it follows that:
\begin{equation*}
\sum_{e_i \in E(\mathcal{H})} |e_i|x_i^2 \leq \sum_{e_i \in E(\mathcal{H})} r x_i^2 = r \sum_{e_i \in E(\mathcal{H})} x_i^2 = r.
\end{equation*}
Substituting this inequality back into the expanded quadratic form, we have $r + \lambda \geq 0$, which implies $\lambda \geq -r$.
\end{proof}

\begin{remark}\label{rem:decomposition}
Proposition \ref{teo:lower_bound} can also be obtained from the clique decomposition method of Cioabă, Elzinga and Gregory \cite[Theorem 2.1 and Corollary 3.1]{smallest-eigenvalue}, see also Knox and Mohar \cite{knox-mohar}. For $k$-uniform hypergraphs, the bound $\lambda \geq -k$ had already been observed in \cite{kaue-lap}, as a direct consequence of the identity of Theorem \ref{teo:multigrafo}. The incidence matrix formulation adopted here has the advantage of connecting the spectrum of $\mathcal{L}(\mathcal{H})$ directly to the signless Laplacian matrix of $\mathcal{H}$, a connection we exploit in Section \ref{sec:App}.
\end{remark}

\begin{lemma}\label{lem:lk}
Let $\mathcal{H}$ be a hypergraph with rank $r$ and $m$ edges. The value $-r$ is an eigenvalue of the adjacency matrix of the line multigraph $\mathcal{L}(\mathcal{H})$ if and only if there exists a non-zero vector $\mathbf{x} \in \mathbb{R}^{m}$ such that:
\begin{itemize}
\item $\mathbf{B}(\mathcal{H})\mathbf{x} = \mathbf{0}$;

\item $x_i = 0$ for every hyperedge $e_i \in E(\mathcal{H})$ with cardinality $|e_i| < r$.
\end{itemize}
\end{lemma}
\begin{proof}
Consider the identity $\mathbf{B}^\mathrm{T}\mathbf{B} = \mathbf{C} + \mathbf{A}_{\mathcal{L}}$. If $\mathbf{x}$ is an eigenvector of the adjacency matrix $\mathbf{A}_{\mathcal{L}}$ associated with the eigenvalue $-r$, then:
\begin{equation*}
\mathbf{x}^\mathrm{T} \mathbf{B}^\mathrm{T} \mathbf{B} \mathbf{x} = \mathbf{x}^\mathrm{T}(\mathbf{C} + \mathbf{A}_{\mathcal{L}})\mathbf{x}\quad \implies\quad ||\mathbf{B}\mathbf{x}||^2 = \sum_{e_i \in E(\mathcal{H})} (|e_i| - r)x_i^2.
\end{equation*}
Since $|e_i| \leq r$ for all $i$, the right-hand side is a sum of non-positive terms, while the left-hand side is non-negative. Thus, equality holds if and only if both sides vanish independently. Specifically, $\|\mathbf{B}\mathbf{x}\|^2 = 0$ implies $\mathbf{B}\mathbf{x} = \mathbf{0}$, and the summation vanishes if and only if $(|e_i| - r)x_i^2 = 0$ for each $i$, which holds if and only if $x_i = 0$ whenever $|e_i| < r$. Conversely, if these two conditions hold, $\mathbf{B}^\mathrm{T}\mathbf{B}\mathbf{x} = \mathbf{0}$ and $(\mathbf{C} - r\mathbf{I})\mathbf{x} = \mathbf{0}$, which directly implies $\mathbf{A}_{\mathcal{L}}\mathbf{x} = -r\mathbf{x}$.
\end{proof}
\begin{remark} 
For a $k$-uniform hypergraph, the second condition is trivial, since $|e_i| = k$ for all $i$.
\end{remark}

\begin{definition}\label{def:subhypergraph}
Let $\mathcal{H} = (V, E)$ and $\mathcal{H}' = (V', E')$ be two hypergraphs. We say that $\mathcal{H}'$ is a \textit{sub-hypergraph} of $\mathcal{H}$ if $V' \subseteq V$ and $E' \subseteq E$. For $\mathcal{H}'$ to be well-defined, every hyperedge $e \in E'$ must satisfy $e \subseteq V'$.
\end{definition}

\begin{definition}\label{def:r_core}
Let $\mathcal{H} = (V, E)$ be a hypergraph with rank $r$. The \textit{$r$-uniform core} of $\mathcal{H}$, denoted by $\mathcal{H}_{[r]}$, is the sub-hypergraph whose hyperedge set is $E_{[r]} = \{e \in E : |e| = r\}$ and whose vertex set consists of all vertices incident to at least one hyperedge in $E_{[r]}$.
\end{definition}

\begin{theorem}\label{teo:collar_eigenvalue}
Let $\mathcal{H}$ be a hypergraph with rank $r$ and $m$ hyperedges. If $\mathcal{H}$ contains an $r$-uniform collar $\mathcal{C}$ as a sub-hypergraph, then $\lambda = -r$ is an eigenvalue of the adjacency matrix $\mathbf{A}_{\mathcal{L}}$ of its line multigraph.
\end{theorem}

\begin{proof}
To prove that $-r$ is an eigenvalue, by Lemma \ref{lem:lk}, it suffices to exhibit a non-zero vector $\mathbf{x} \in \mathbb{R}^{m}$ with $\mathbf{B}\mathbf{x} = \mathbf{0}$ and $x_i = 0$ for every hyperedge $e_i$ with $|e_i| < r$.

Let $E(\mathcal{C}) \subseteq E(\mathcal{H})$ be the set of hyperedges in the collar sub-hypergraph; since $\mathcal{C}$ is $r$-uniform, every hyperedge in $E(\mathcal{C})$ has cardinality $r$, so $E(\mathcal{C}) \subseteq E_{[r]}$. Since $\mathcal{C}$ is a collar, it admits an edge $2$-coloring $f: E(\mathcal{C}) \to \{1, 2\}$. We define the vector $\mathbf{x} \in \mathbb{R}^m$ as:
$$x_i = \begin{cases} 
\;\;\,1 & \text{if } e_i \in E(\mathcal{C}) \text{ and } f(e_i) = 1, \\ 
-1 & \text{if } e_i \in E(\mathcal{C}) \text{ and } f(e_i) = 2, \\
\;\;\,0 & \text{if } e_i \notin E(\mathcal{C}).
\end{cases}$$
Since $x_i \neq 0$ only for $e_i \in E(\mathcal{C})$, and all such hyperedges have cardinality $r$, the second condition of Lemma \ref{lem:lk} holds.

Now, we examine the entries of the product $\mathbf{B}\mathbf{x}$ for each vertex $v \in V(\mathcal{H})$:
\begin{itemize}
\item \textbf{Case 1 ($v \notin V(\mathcal{C})$):} If the vertex $v$ is not part of the collar, it is not incident to any hyperedge $e_i$ for which $x_i \neq 0$. Thus, $(\mathbf{B}\mathbf{x})_v = 0$.
\item \textbf{Case 2 ($v \in V(\mathcal{C})$):} Since $\mathcal{C}$ is a $2$-regular sub-hypergraph, $v$ is incident to exactly two hyperedges in $E(\mathcal{C})$. By the definition of the edge $2$-coloring of a collar, one of these hyperedges must have color $1$ ($x_i = 1$) and the other must have color $2$ ($x_j = -1$). Consequently, $(\mathbf{B}\mathbf{x})_v = 1 + (-1) = 0.$
\end{itemize}
In both cases, we find that $\mathbf{B}\mathbf{x} = \mathbf{0}$. Since $\mathbf{x}$ is a non-zero vector satisfying both conditions of Lemma \ref{lem:lk}, it follows that $\lambda = -r$ is an eigenvalue of $\mathbf{A}_{\mathcal{L}}$.
\end{proof}

Theorem \ref{teo:collar_eigenvalue} gives a sufficient condition for the lower bound of Proposition \ref{teo:lower_bound} to be attained. We now determine exactly when $-r$ is an eigenvalue of $\mathbf{A}_{\mathcal{L}}$, and with which multiplicity. The idea is to isolate the part of the $r$-uniform core that can support an eigenvector of $-r$. Our construction is the hypergraph counterpart of the reduction to $\mathcal{K}$-essential vertices developed by Cioabă, Elzinga and Gregory \cite{smallest-eigenvalue} for clique partitions.

\begin{definition}\label{def:essential_core}
Let $\mathcal{H}$ be a hypergraph with rank $r$ and $r$-uniform core $\mathcal{H}_{[r]}$. A family $F \subseteq E_{[r]}$ is \textit{doubly covering} if every vertex belonging to a member of $F$ belongs to at least two members of $F$. Since the union of doubly covering families is doubly covering, there exists a unique maximal doubly covering family, denoted by $E^*$. The \textit{essential core} of $\mathcal{H}$ is the sub-hypergraph $\mathcal{H}^* = (V^*, E^*)$, where $V^* = \bigcup_{e \in E^*} e$. The essential core, which may be empty, is a sub-hypergraph of the $r$-uniform core.
\end{definition}

\begin{lemma}\label{lem:peeling}
The family $E^*$ is produced by the following peeling procedure: start from $E_{[r]}$ and repeatedly delete a hyperedge containing a vertex that belongs to no other hyperedge of the current family, until no such hyperedge exists. In particular, the outcome of the procedure does not depend on the order of the deletions.
\end{lemma}
\begin{proof}
Let $F$ be the family produced by a complete run of the procedure. By the stopping rule, every vertex covered by $F$ belongs to at least two members of $F$, so $F$ is doubly covering and, by maximality, $F \subseteq E^*$. Conversely, we claim that no member of $E^*$ is ever deleted. Indeed, suppose inductively that the current family contains $E^*$; this holds at the start, since $E^* \subseteq E_{[r]}$. If $e \in E^*$, then every vertex $v \in e$ belongs to a second hyperedge of $E^*$, which is still present in the current family; hence $e$ is not deletable, and the induction continues. Therefore $E^* \subseteq F$, and $F = E^*$.
\end{proof}

\begin{theorem}\label{teo:multiplicity}
Let $\mathcal{H}$ be a hypergraph with rank $r$, and let $\mathcal{H}^* = (V^*, E^*)$ be its essential core, with incidence matrix $\mathbf{B}^* = \mathbf{B}(\mathcal{H}^*)$. For $\mathbf{y} \in \mathbb{R}^{|E^*|}$, let $\widetilde{\mathbf{y}} \in \mathbb{R}^{|E|}$ denote its \textit{extension by zeros}. Then $\mathbf{y} \mapsto \widetilde{\mathbf{y}}$ is a linear isomorphism from the null space of $\mathbf{B}^*$ onto the eigenspace of $\mathbf{A}_{\mathcal{L}}$ associated with $-r$. In particular, the multiplicity $\mu$ of $-r$ as an eigenvalue of $\mathbf{A}_{\mathcal{L}}$ is given by
\begin{equation*}
\mu = |E^*| - \operatorname{rank}(\mathbf{B}^*),
\end{equation*}
and $-r$ is an eigenvalue of $\mathbf{A}_{\mathcal{L}}$ if and only if the columns of $\mathbf{B}^*$ are linearly dependent.
\end{theorem}
\begin{proof}
As shown in the proof of Lemma \ref{lem:lk}, a vector $\mathbf{x} \in \mathbb{R}^{|E|}$ satisfies $\mathbf{A}_{\mathcal{L}}\mathbf{x} = -r\mathbf{x}$ if and only if $\mathbf{B}\mathbf{x} = \mathbf{0}$ and $x_e = 0$ for every hyperedge $e$ with $|e| < r$. We show that the set of such vectors is exactly the image, under the extension by zeros, of the null space of $\mathbf{B}^*$.

First, let $\mathbf{y} \in \mathbb{R}^{|E^*|}$ satisfy $\mathbf{B}^*\mathbf{y} = \mathbf{0}$, and let $\mathbf{x} \in \mathbb{R}^{|E|}$ be its extension by zeros. Then $x_e = 0$ whenever $|e| < r$, because $E^* \subseteq E_{[r]}$. Moreover, for every vertex $v \in V(\mathcal{H})$, the hyperedges containing $v$ that lie outside $E^*$ contribute nothing to $(\mathbf{B}\mathbf{x})_v$; hence $(\mathbf{B}\mathbf{x})_v = 0$ if $v \notin V^*$, while $(\mathbf{B}\mathbf{x})_v = (\mathbf{B}^*\mathbf{y})_v = 0$ if $v \in V^*$. Thus $\mathbf{B}\mathbf{x} = \mathbf{0}$.

Conversely, let $\mathbf{x}$ satisfy $\mathbf{B}\mathbf{x} = \mathbf{0}$ and $x_e = 0$ for $|e| < r$. We claim that $x_e = 0$ for every hyperedge deleted by the peeling procedure of Lemma \ref{lem:peeling}. We argue by induction on the deletion steps. For the first deletion, the claim is vacuously true, since no hyperedge has been deleted yet; moreover, if $e$ is the first hyperedge deleted because a vertex $v \in e$ belongs to no other member of $E_{[r]}$, then every hyperedge $f \neq e$ containing $v$ already satisfies $|f| < r$, so $x_f = 0$ by hypothesis. For the inductive step, suppose the claim holds for all previously deleted hyperedges, and let $e$ be deleted because some vertex $v \in e$ belongs to no other member of the current family. Every hyperedge $f \neq e$ containing $v$ either has cardinality smaller than $r$, or was deleted earlier; in both cases $x_f = 0$. Therefore $0 = (\mathbf{B}\mathbf{x})_v = x_e$, which proves the claim. Consequently, $\mathbf{x}$ is supported on $E^*$ and, denoting by $\mathbf{y}$ its restriction to $E^*$, the same computation as above gives $(\mathbf{B}^*\mathbf{y})_v = (\mathbf{B}\mathbf{x})_v = 0$ for all $v \in V^*$.

The extension by zeros is injective by construction, and the argument above shows that its image is exactly the eigenspace. The dimension formula then follows from the rank-nullity theorem.
\end{proof}

\begin{corollary}\label{cor:collar_general}

Let $\mathcal{H}$ be a hypergraph with rank $r$, and suppose $\mathcal{H}$ contains a family of $r$-uniform collars $\mathcal{C}_1, \dots, \mathcal{C}_t$ as sub-hypergraphs. For each $i \in \{1,\dots,t\}$, let $\mathbf{x}^{(i)} \in \mathbb{R}^{|E(\mathcal{H})|}$ be the vector equal to $+1$ on the odd hyperedges of $\mathcal{C}_i$, $-1$ on its even hyperedges, and $0$ elsewhere. If $\mathbf{x}^{(1)}, \dots, \mathbf{x}^{(t)}$ are linearly independent, then the multiplicity $\mu$ of $-r$ satisfies $\mu \geq t$. In particular, a single $r$-uniform collar already forces $-r$ to be an eigenvalue of the line multigraph.

\end{corollary}

\begin{proof}
Each set $E(\mathcal{C}_i)$ consists of hyperedges of cardinality $r$ and, since a collar is $2$-regular, every vertex of $\mathcal{C}_i$ belongs to exactly two members of $E(\mathcal{C}_i)$; hence $E(\mathcal{C}_i)$ is doubly covering and $E(\mathcal{C}_i) \subseteq E^*$. In particular, each $\mathbf{x}^{(i)}$ vanishes outside $E^*$; let $\mathbf{y}^{(i)} \in \mathbb{R}^{|E^*|}$ denote its restriction to $E^*$, so that $\mathbf{x}^{(i)} \mapsto  \mathbf{y}^{(i)}$ and the vectors $\mathbf{y}^{(1)}, \dots, \mathbf{y}^{(t)}$ are linearly independent as well. For $v\in V(\mathcal{C}_i)$, the two hyperedges of $\mathcal{C}_i$ containing $v$ have opposite colors, so they contribute $\pm1$ to $(\mathbf{B}^*\mathbf{y}^{(i)})_v$; for any other $v \in V^*$, no hyperedge of $\mathcal{C}_i$ contains $v$, so this entry is $0$. Hence $\mathbf{B}^*\mathbf{y}^{(i)} = \mathbf{0}$ for every $i$. The result now follows from Theorem \ref{teo:multiplicity}.
\end{proof}

\begin{remark}\label{rem:doob}
For a graph $\mathcal{G}$, seen as a $2$-uniform hypergraph, the peeling procedure of Lemma \ref{lem:peeling} repeatedly removes an edge incident to a vertex of degree $1$, that is, a pendant edge, until none remain. Under this reduction, a tree collapses to the empty hypergraph, while a unicyclic graph collapses to its unique cycle. Combining this with Theorem \ref{teo:multiplicity}, the multiplicity $\mu$ of $-2$ vanishes if and only if the essential core $\mathcal{G}^*$ is empty or an odd cycle, that is, if and only if $\mathcal{G}$ is a tree or an odd-unicyclic graph. Thus $\lambda(\mathcal{L}(\mathcal{G})) > -2$ exactly in this case, recovering the classical characterization for line graphs \cite{line-spec}.
\end{remark}

\begin{remark}\label{rem:collar_basis}
In a general hypergraph, collars play the role that even cycles play in a graph: configurations that push the spectrum of the line multigraph down to its smallest possible value. However, collar vectors do not span the eigenspace of $-r$ in general. Already for graphs, the eigenvectors of $-2$ of a line graph come from even cycles and from pairs of odd cycles joined by a path \cite{line-spec}, and these paths are not collars. Theorem \ref{teo:multiplicity} still gives the multiplicity of $-r$, through the rank of the essential core. But a simple combinatorial description of a basis for the eigenspace, one that finds the hypergraph version of even cycles and of pairs of odd cycles joined by a path, remains open; see Section \ref{sec:conclusion}.
\end{remark}

\section{Applications to spectral hypergraph theory}\label{sec:App}

This section applies the theory developed so far to the signless Laplacian matrix $\mathbf{Q}(\mathcal{H}) = \mathbf{B}\mathbf{B}^\mathrm{T}$. Its spectrum is closely related to the spectrum of the line multigraph, because both matrices come from the same incidence matrix $\mathbf{B}$. A detailed treatment of this connection, restricted to uniform hypergraphs, can be found in \cite{kaue-energia,kaue-lap}.

\begin{definition}\label{def:signless_laplacian}
Let $\mathcal{H}$ be a hypergraph with incidence matrix $\mathbf{B}$. The \textit{signless Laplacian matrix} of $\mathcal{H}$ is defined as the $|V| \times |V|$ matrix:$$\mathbf{Q}(\mathcal{H}) = \mathbf{B}\mathbf{B}^\mathrm{T}.$$For an irreducible and non-negative matrix, the \textit{spectral radius} $\rho(\mathcal{H})$ is defined as its largest eigenvalue, as guaranteed by the Perron-Frobenius Theorem.
\end{definition}
\begin{remark}
The entries of $\mathbf{Q}(\mathcal{H}) = (q_{uv})$ have a direct combinatorial interpretation: the diagonal entries $q_{vv}$ are the vertex degrees $d_{\mathcal{H}}(v)$, while the off-diagonal entries $q_{uv}$ represent the number of hyperedges in $\mathcal{H}$ that contain both vertices $u$ and $v$.
\end{remark}

\begin{theorem}\label{teo:raio_espectral}
Let $\mathcal{H}$ be a hypergraph with rank $r$ and co-rank $s$. If $\mathbf{A}_{\mathcal{L}}$ is the adjacency matrix of its line multigraph, then:
$$\rho(\mathbf{Q}) - r \leq \rho(\mathbf{A}_{\mathcal{L}}) \leq \rho(\mathbf{Q}) - s.$$
Furthermore, either equality holds if and only if $\mathcal{H}$ is uniform.
\end{theorem}

\begin{proof}
First, observe that $\mathbf{Q} = \mathbf{B}\mathbf{B}^\mathrm{T}$ and $\mathbf{B}^\mathrm{T}\mathbf{B}$ share the same spectral radius. Let $\mathbf{x} \in \mathbb{R}^{|E|}$ be a normalized vector. Since $s$ and $r$ are the minimum and maximum diagonal entries of $\mathbf{C}$, respectively, we have
\begin{equation*}
s \leq \mathbf{x}^\mathrm{T}\mathbf{C}\mathbf{x} \leq r.
\end{equation*}

Let $\mathbf{u}$ be the normalized principal eigenvector of $\mathbf{B}^\mathrm{T}\mathbf{B}$. Then:$$\rho(\mathbf{Q}) = \mathbf{u}^\mathrm{T}(\mathbf{C} + \mathbf{A}_{\mathcal{L}})\mathbf{u} = \mathbf{u}^\mathrm{T}\mathbf{C}\mathbf{u} + \mathbf{u}^\mathrm{T}\mathbf{A}_{\mathcal{L}}\mathbf{u} \leq r + \rho(\mathbf{A}_{\mathcal{L}}).$$

Let $\mathbf{v}$ be the normalized principal eigenvector of $\mathbf{A}_{\mathcal{L}}$. Then:
$$\rho(\mathbf{Q}) \geq \mathbf{v}^\mathrm{T}(\mathbf{C} + \mathbf{A}_{\mathcal{L}})\mathbf{v} = \mathbf{v}^\mathrm{T}\mathbf{C}\mathbf{v} + \rho(\mathbf{A}_{\mathcal{L}}) \geq s + \rho(\mathbf{A}_{\mathcal{L}}).$$

If $\mathcal{H}$ is uniform, then $\mathbf{C} = r\mathbf{I} = s\mathbf{I}$ and both bounds hold with equality. Conversely, suppose that the upper bound holds with equality, that is, $\rho(\mathbf{Q}) = r + \rho(\mathbf{A}_{\mathcal{L}})$. Since $\mathbf{u}^\mathrm{T}\mathbf{C}\mathbf{u} \leq r$ and $\mathbf{u}^\mathrm{T}\mathbf{A}_{\mathcal{L}}\mathbf{u} \leq \rho(\mathbf{A}_{\mathcal{L}})$, both inequalities must hold with equality; in particular, $\mathbf{u}^\mathrm{T}\mathbf{C}\mathbf{u} = \sum_{i} |e_i| u_i^2 = r$. As $\mathcal{H}$ is connected, the non-negative matrix $\mathbf{B}^\mathrm{T}\mathbf{B} = \mathbf{C} + \mathbf{A}_{\mathcal{L}}$ is irreducible, so all entries of its principal eigenvector $\mathbf{u}$ are positive, by the Perron-Frobenius Theorem. Hence $\sum_{i} (r - |e_i|)u_i^2 = 0$ forces $|e_i| = r$ for every $i$, and $\mathcal{H}$ is uniform. The same argument, applied to the principal eigenvector $\mathbf{v}$ of the irreducible non-negative matrix $\mathbf{A}_{\mathcal{L}}$, shows that equality in the lower bound forces $|e_i| = s$ for every $i$.
\end{proof}

For uniform hypergraphs, both bounds of Theorem \ref{teo:raio_espectral} collapse into the exact identity $\rho(\mathbf{A}_{\mathcal{L}}) = \rho(\mathbf{Q}) - k$, observed in \cite{kaue-lap}. Here, we generalize this identity to general hypergraphs, showing how it becomes a pair of bounds, controlled by the rank and the co-rank, when uniformity is dropped.

\begin{theorem}\label{teo:qcotas}
Let $\mathcal{H}$ be a hypergraph with rank $r$ and co-rank $s$. If $\rho(\mathbf{Q})$ denotes the spectral radius of its signless Laplacian matrix, then:
\begin{equation*}
\min_{e\in E} \left\{ \sum_{v \in e} d(v) \right\} - (r-s) \leq \rho(\mathbf{Q}) \leq \max_{e\in E} \left\{ \sum_{v \in e} d(v) \right\} + (r-s).
\end{equation*}
Either equality holds if and only if $\mathcal{H}$ is both uniform and edge-regular.
\end{theorem}

\begin{proof}
Let $u \in V(\mathcal{L}(\mathcal{H}))$ be a vertex in the line multigraph corresponding to the hyperedge $e_u \in E(\mathcal{H})$. By Lemma \ref{lem:graulinha}, the degree of $u$ in the line multigraph is:
\begin{equation*}
d_{\mathcal{L}}(u) = \left( \sum_{v \in e_u} d_{\mathcal{H}}(v) \right) - |e_u|.
\end{equation*}
From Theorem \ref{teo:raio_espectral}, we have the following bounds for the spectral radius:
\begin{equation}\label{eq:lema-rho-refined}
\rho(\mathbf{Q}) - r \leq \rho(\mathbf{A}_{\mathcal{L}}) \leq \rho(\mathbf{Q}) - s.
\end{equation}
It is a well-established result that the spectral radius of the adjacency matrix of a connected graph lies between its minimum and maximum degrees, with either equality only in the regular case \cite{Stevanovic}; the Perron-Frobenius argument extends this statement to connected multigraphs:
\begin{equation}\label{eq:multigraph-refined}
\min_{u \in V(\mathcal{L})} d_{\mathcal{L}}(u) \leq \rho(\mathbf{A}_{\mathcal{L}}) \leq \max_{u \in V(\mathcal{L})} d_{\mathcal{L}}(u).
\end{equation}
Combining the lower bound of \eqref{eq:lema-rho-refined} with the upper bound of \eqref{eq:multigraph-refined}, we obtain:
\begin{equation*}
\rho(\mathbf{Q}) - r \leq \max_{e \in E} \left\{ \left( \sum_{v \in e} d_{\mathcal{H}}(v) \right) - |e| \right\} \leq \max_{e \in E} \left\{ \sum_{v \in e} d_{\mathcal{H}}(v) \right\} - s.
\end{equation*}
Similarly, for the lower bound:
\begin{equation*}
\rho(\mathbf{Q}) - s \geq \min_{e \in E} \left\{ \left( \sum_{v \in e} d_{\mathcal{H}}(v) \right) - |e| \right\} \geq \min_{e \in E} \left\{ \sum_{v \in e} d_{\mathcal{H}}(v) \right\} - r.
\end{equation*}
If $\mathcal{H}$ is uniform and edge-regular, then $|e| = r = s$ for every hyperedge and, by Lemma \ref{lem:uniform_equiv} and Theorem \ref{teo:arestareg}, $\mathcal{L}(\mathcal{H})$ is regular; hence all the inequalities above collapse into the equality $\rho(\mathbf{Q}) = \sum_{v \in e} d(v)$ for any $e \in E(\mathcal{H})$.

Conversely, suppose that the upper bound holds with equality. Then each inequality in the chain above must hold with equality: (i) $\rho(\mathbf{Q}) - r = \rho(\mathbf{A}_{\mathcal{L}})$, which, by Theorem \ref{teo:raio_espectral}, implies that $\mathcal{H}$ is uniform; (ii) $\rho(\mathbf{A}_{\mathcal{L}}) = \max_{u} d_{\mathcal{L}}(u)$, which, for a connected multigraph, holds if and only if $\mathcal{L}(\mathcal{H})$ is regular \cite{Stevanovic}; and (iii) $\max_{e}\left\{\sum_{v \in e} d(v) - |e|\right\} = \max_{e}\left\{\sum_{v \in e} d(v)\right\} - s$, which is automatic once $\mathcal{H}$ is uniform, since then $|e| = s$ for every $e$. By Theorem \ref{teo:arestareg} and Lemma \ref{lem:uniform_equiv}, a uniform hypergraph with regular line multigraph is edge-regular. The analysis of the lower bound is analogous, using that the spectral radius of a connected multigraph equals its minimum degree only in the regular case \cite{Stevanovic}. Thus, either equality holds if and only if $\mathcal{H}$ is uniform and edge-regular.
\end{proof}

\begin{definition}\label{def:general_power}
Let $\mathcal{H}=(V,E)$ be a hypergraph with rank $r$. For integers $t \geq 1$ and $q \geq 0$, the \textit{general power hypergraph} $\mathcal{H}^k_t$ is the hypergraph with rank $k = tr + q$ constructed from $\mathcal{H}$ as follows:
\begin{itemize}
    \item \textbf{Vertex Expansion:} Each vertex $v \in V(\mathcal{H})$ is replaced by a set $\tau_v$ of cardinality $t$. Each vertex in $\tau_v$ inherits the incidences of the original vertex $v$. Specifically, for any $e \in E(\mathcal{H})$, the corresponding hyperedge $e_t \in E(\mathcal{H}_t)$ contains the entire set $\tau_v$ if and only if $v \in e$.
    
    \item \textbf{Hyperedge Padding:} For each hyperedge $e_t \in E(\mathcal{H}_t)$, a set $\tau_e$ of $q = k - rt$ new distinct vertices is added to the corresponding hyperedge $e^k_t$. Each vertex in $\tau_e$ is defined to have degree one in $\mathcal{H}^k_t$.
\end{itemize}
\end{definition}

When $t = 1$, we write $\mathcal{H}^k := \mathcal{H}^k_1$ (hyperedge padding only); when $k = rt$, we write $\mathcal{H}_t := \mathcal{H}^{rt}_t$ (vertex expansion only). Observe that these operations compose: $\mathcal{H}^k_t = (\mathcal{H}_t)^k$, that is, every general power hypergraph is obtained by first expanding the vertices and then padding the hyperedges.

The signless Laplacian spectrum of power hypergraphs was previously investigated for the uniform-base case in \cite{kaue-lap}, and the behavior of line multigraphs of uniform hypergraphs under power operations was described in \cite{kaue-energia}. Here, applying line multigraph theory, we extend both to general base hypergraphs and provide a significantly more concise proof. By analyzing the structural transformation of the line multigraph, specifically the scaling of its adjacency matrix, we avoid the combinatorial complexities inherent to non-uniform incidence structures.

\begin{corollary}\label{cor:power_inv}
Let $\mathcal{H}$ be a hypergraph with rank $r$. For any integer $k \geq r$, the line multigraph of the power hypergraph $\mathcal{H}^k$ satisfies $\mathcal{L}(\mathcal{H}^k) = \mathcal{L}(\mathcal{H})$.
\end{corollary}
\begin{proof}
The hyperedge padding step in the construction of $\mathcal{H}^k$ (Definition \ref{def:general_power}) adjoins degree-one vertices to each hyperedge, exactly as in Lemma \ref{lem:remove-vertex} and Corollary \ref{cor:uniform_representation}: since a degree-one vertex belongs to no other hyperedge, it cannot contribute to any intersection $e_i \cap e_j$ with $i \neq j$. All such intersections, and hence $\mathcal{L}(\mathcal{H})$ itself, are therefore unchanged by the padding.
\end{proof}

\begin{definition}\label{def:scaled_multigraph}
Let $\mathcal{G}$ be a multigraph and $t \geq 1$ be an integer. The \textit{scaled multigraph} $t \cdot \mathcal{G}$ is the multigraph obtained from $\mathcal{G}$ as follows:
\begin{itemize}
\item The vertex set is preserved: $V(t \cdot \mathcal{G}) = V(\mathcal{G})$;
\item The edge multiplicities are scaled by $t$: if a pair of vertices is connected by $p$ edges in $\mathcal{G}$, then they are connected by $tp$ edges in $t \cdot \mathcal{G}$.
\end{itemize}
\end{definition}

\begin{lemma}\label{lem:linhk}
	Let $\mathcal{H}$ be a hypergraph. If $t \geq 1$ is an integer, then $\mathcal{L}(\mathcal{H}_t) = t \cdot \mathcal{L}(\mathcal{H})$.
\end{lemma}
\begin{proof}
By construction, the hyperedges of $\mathcal{H}_t$ are in one-to-one correspondence with those of $\mathcal{H}$, implying $V(\mathcal{L}(\mathcal{H}_t)) = V(\mathcal{L}(\mathcal{H}))$. Let $e_i, e_j \in E(\mathcal{H})$ be two distinct hyperedges with an intersection of cardinality $p = |e_i \cap e_j|$. In the expanded hypergraph $\mathcal{H}_t$, each vertex $v \in e_i \cap e_j$ is replaced by a set $\tau_v$ of $t$ vertices, all of which are contained in the intersection of the corresponding hyperedges $(e_i)_t, (e_j)_t \in E(\mathcal{H}_t)$. Consequently, $|(e_i)_t \cap (e_j)_t| = \sum_{v \in e_i \cap e_j} |\tau_v| = t \cdot p$. By the definition of a line multigraph, the multiplicity of the edge connecting $(e_i)_t$ and $(e_j)_t$ in $\mathcal{L}(\mathcal{H}_t)$ is exactly $t$ times the multiplicity of the edge connecting $e_i$ and $e_j$ in $\mathcal{L}(\mathcal{H})$. This satisfies the definition of the scaled multigraph, thus $\mathcal{L}(\mathcal{H}_t) = t \cdot \mathcal{L}(\mathcal{H})$.\end{proof}

\begin{remark}\label{rem:scal-matrix}
Since $\mathcal{H}^k_t = (\mathcal{H}_t)^k$, Corollary \ref{cor:power_inv} and Lemma \ref{lem:linhk} together imply the matrix identity $\mathbf{A}_{\mathcal{L}}(\mathcal{H}^k_t) = t\,\mathbf{A}_{\mathcal{L}}(\mathcal{H})$. Observe that this scaling depends entirely on the vertex expansion factor and is independent of edge padding.
\end{remark}

For a square matrix $\mathbf{M}$, we denote by $\mathsf{P}_{\mathbf{M}}(\lambda) = \det(\lambda\mathbf{I} - \mathbf{M})$ its characteristic polynomial; for a multigraph $\mathcal{G}$, we abbreviate $\mathsf{P}_{\mathcal{G}} := \mathsf{P}_{\mathbf{A}(\mathcal{G})}$.

\begin{proposition}\label{teo:espec-power-lin}
Let $\mathcal{H}$ be a hypergraph with rank $r$ and $m$ edges. For integers $t \geq 1$ and $k \geq rt$, let $\mathcal{H}^k_t$ be its general power hypergraph. Then:
\begin{equation*}
\mathsf{P}_{\mathcal{L}(\mathcal{H}^k_t)}(\lambda) = t^m \mathsf{P}_{\mathcal{L}(\mathcal{H})}(\lambda/t).
\end{equation*}
Consequently, $\lambda$ is an eigenvalue of $\mathbf{A}_{\mathcal{L}}(\mathcal{H})$ if and only if $t\lambda$ is an eigenvalue of $\mathbf{A}_{\mathcal{L}}(\mathcal{H}^k_t)$.
\end{proposition}
\begin{proof}
By applying Remark \ref{rem:scal-matrix}, we have the scaling matrix identity $\mathbf{A}_{\mathcal{L}}(\mathcal{H}^k_t) = t \mathbf{A}_{\mathcal{L}}(\mathcal{H})$. The characteristic polynomial is then:
\begin{align*}
\mathsf{P}_{\mathcal{L}(\mathcal{H}^k_t)}(\lambda) &= \det\left(\lambda\mathbf{I}_m - \mathbf{A}(\mathcal{L}(\mathcal{H}^k_t)) \right) \\
&= \det\left(\lambda\mathbf{I}_m - t\mathbf{A}(\mathcal{L}(\mathcal{H})) \right) \\
&= t^m \det\left((\lambda/t)\mathbf{I}_m - \mathbf{A}(\mathcal{L}(\mathcal{H})) \right) \\
&= t^m \mathsf{P}_{\mathcal{L}(\mathcal{H})}(\lambda/t).
\end{align*}
The linear relationship between the eigenvalues follows directly from the roots of the transformed characteristic polynomial. For a uniform base hypergraph, this transformation of the characteristic polynomial was proved in \cite{kaue-energia}; Corollary \ref{cor:power_inv} and Lemma \ref{lem:linhk} are what make the extension to general base hypergraphs immediate.
\end{proof}

\begin{theorem}\label{teo:power_spectrum}
Let $\mathcal{H}$ be a hypergraph with $n$ vertices, $m$ edges and rank $r$, and let $\lambda_1 \geq \lambda_2 \geq \cdots \geq \lambda_n$ be the eigenvalues of its signless Laplacian matrix, where $\lambda_i > 0$ for $i \leq p$ and $\lambda_i = 0$ for $i > p$. For integers $t \geq 1$ and $k > rt$, the eigenvalues of the signless Laplacian matrix of the power hypergraph $\mathcal{H}^k_t$ are:
\begin{itemize}
    \item[i.] $(t\lambda_1+(k-rt)),\; \ldots,\; (t\lambda_p+(k-rt))$,
    \item[ii.] $(k-rt)$ with multiplicity $m-p$,
    \item[iii.] $0$ with multiplicity $(k-rt-1)m+tn$.
\end{itemize}
\end{theorem}

\begin{proof}
By the definition of a power hypergraph, $|V(\mathcal{H}^k_t)| = tn + m(k-rt)$; since $k > rt$, this number is at least $m$. Since $\mathbf{B}\mathbf{B}^\mathrm{T}$ and $\mathbf{B}^\mathrm{T}\mathbf{B}$ share the same non-zero eigenvalues, we have:
\begin{align*}
\mathsf{P}_{\mathbf{Q}(\mathcal{H}^k_t)}(\lambda) 
&= \det(\lambda\mathbf{I}_{tn+m(k-rt)}-\mathbf{B}(\mathcal{H}^k_t)\mathbf{B}^\mathrm{T}(\mathcal{H}^k_t))\\
&= \lambda^{tn+m(k-rt)-m}\det(\lambda\mathbf{I}_m - \mathbf{B}^\mathrm{T}(\mathcal{H}^k_t)\mathbf{B}(\mathcal{H}^k_t))\\
&= \lambda^{tn+(k-rt-1)m}\det\left(\lambda\mathbf{I}_m - \big[\mathbf{C}(\mathcal{H}^k_t)\big] - \mathbf{A}_{\mathcal{L}}(\mathcal{H}^k_t) \right) \\
&= \lambda^{tn+(k-rt-1)m}\det\left(\lambda\mathbf{I}_m - \big[(k-rt)\mathbf{I}_m + t\mathbf{C}(\mathcal{H})\big]- t\mathbf{A}_{\mathcal{L}}(\mathcal{H}) \right)\\
&= \lambda^{tn+(k-rt-1)m}\det\left([\lambda-(k-rt)]\mathbf{I}_m - t[\mathbf{C}(\mathcal{H})+\mathbf{A}_{\mathcal{L}}(\mathcal{H})] \right)\\
&= \lambda^{tn+(k-rt-1)m}\det\left([\lambda-(k-rt)]\mathbf{I}_m - t[\mathbf{B}^\mathrm{T}(\mathcal{H})\mathbf{B}(\mathcal{H})] \right)\\
&= \lambda^{tn+(k-rt-1)m} t^m\det\left(\left[ \frac{\lambda-(k-rt)}{t}\right] \mathbf{I}_m -\mathbf{B}^\mathrm{T}(\mathcal{H})\mathbf{B}(\mathcal{H}) \right).
\end{align*}
Therefore, for each non-zero eigenvalue $\lambda_i$ of $\mathbf{B}^\mathrm{T}(\mathcal{H})\mathbf{B}(\mathcal{H})$, we obtain the eigenvalue $\mu_i = t\lambda_i + (k-rt)$ for $\mathbf{Q}(\mathcal{H}^k_t)$. For each of the $(m-p)$ zero eigenvalues of $\mathbf{B}^\mathrm{T}\mathbf{B}$, we obtain the eigenvalue $(k-rt)$. Finally, the factor $\lambda^{tn+m(k-rt-1)}$ accounts for the multiplicity of the eigenvalue $0$.
\end{proof}

\begin{remark}
When $k = rt$, that is, when no padding occurs, item (iii) is no longer meaningful as stated, since the exponent $tn - m$ may be negative. In this case, a direct computation using Theorem \ref{teo:multigrafo} and Remark \ref{rem:scal-matrix} gives $\mathbf{B}^\mathrm{T}(\mathcal{H}_t)\mathbf{B}(\mathcal{H}_t) = t\,\mathbf{B}^\mathrm{T}(\mathcal{H})\mathbf{B}(\mathcal{H})$, so the non-zero eigenvalues of $\mathbf{Q}(\mathcal{H}_t)$ are $t\lambda_1, \ldots, t\lambda_p$, and $0$ is an eigenvalue with multiplicity $tn - p$, which is consistent with $|V(\mathcal{H}_t)| = tn$.
\end{remark}

The results of this section follow a common pattern: once the incidence identity of Theorem \ref{teo:multigrafo} is available, a spectral result can often be extended from uniform to general hypergraphs by tracking how the cardinality matrix $\mathbf{C}$ replaces the scalar matrix $k\mathbf{I}$. The spectrum of general power hypergraphs is the clearest example of this pattern.

\section{Conclusion}\label{sec:conclusion}

In this paper, we develop the theory of line multigraphs for general hypergraphs. Using the incidence identity $\mathbf{B}^\mathrm{T}\mathbf{B} = \mathbf{C} + \mathbf{A}_{\mathcal{L}}$, we generalize most of the known results on this topic and prove several new ones. On the structural side, we showed that connectivity, linearity, and regularity transfer between a hypergraph and its line multigraph, and we described how much information the correspondence loses: every connected multigraph in which each vertex has at least two distinct neighbors is a line multigraph (Proposition \ref{prop:all_multigraphs}), and every hypergraph shares its line multigraph with infinitely many others (the line-equivalence relation of Section \ref{sec:structural}). On the spectral side, we proved that the eigenvalues of the line multigraph of a hypergraph of rank $r$ are at least $-r$, we determined the eigenspace and the multiplicity of $-r$ as an eigenvalue, and we showed that the presence of a collar is a sufficient combinatorial condition for $-r$ to be attained. We then applied these tools to bound the spectral radius of the signless Laplacian matrix and to determine the complete signless Laplacian spectrum of general power hypergraphs, extending results previously available only in the uniform case.

Some of the proofs given here for general hypergraphs are shorter than the original proofs for the uniform case. This is possible because the theory of line multigraphs condenses much of the hypergraph's information into a single framework, in place of the separate, case-by-case constructions that earlier proofs required. The open problems below record the points where our results are not yet complete.

\begin{enumerate}
\item Theorem \ref{teo:multiplicity} computes the multiplicity of $-r$ in $\mathcal{L}(\mathcal{H})$ as the nullity of the incidence matrix of the essential core, and Corollary \ref{cor:collar_general} shows that linearly independent collars provide a lower bound for it. Can one describe a combinatorial basis for the eigenspace of $-r$, identifying the hypergraph analogues of even cycles and of pairs of odd cycles joined by a path?

\item By Theorem \ref{teo:multiplicity}, $\lambda(\mathcal{L}(\mathcal{H})) > -r$ if and only if the incidence matrix of the essential core of $\mathcal{H}$ has linearly independent columns. Can this rank condition be replaced by a purely structural description, identifying the hypergraph analogues of trees and odd-unicyclic graphs?
\end{enumerate}

\bibliographystyle{siamplain}
\bibliography{Bibliografia}

\end{document}